\documentclass[12pt]{amsart}

\usepackage{amsfonts}
\usepackage{amssymb}
\usepackage{amsmath}
\usepackage{amscd}
\usepackage{amstext}
\usepackage{ifthen}
\usepackage[all]{xy}
\usepackage{enumerate}
\usepackage[pdftex]{color}

\makeatletter
\def\@cite#1#2{{\m@th\upshape\bfseries%
[{#1\if@tempswa{\m@th\upshape\mdseries, #2}\fi}]}}
\makeatother

\newtheorem{theorem}{Theorem}[section]
\newtheorem{lemma}[theorem]{Lemma}
\newtheorem{corollary}[theorem]{Corollary}
\newtheorem{proposition}[theorem]{Proposition}

\theoremstyle{definition}

\newtheorem{remark}[theorem]{Remark}

\numberwithin{equation}{section}

\renewcommand{\phi}{\varphi}
\newcommand{\upchi}{{\raise.35ex\hbox{\ensuremath{\chi}}}}

\begin{document}
%%%%%%%%%%%%%%%%%%%%%%%%%%%%%%%%%%%%%%%%%%%

\title{Subalgebras of $C(\Omega,M_n)$ and their modules}
\author{Jean Roydor}
\address{Departement de Math\'{e}matiques\\
Universit\'{e} de Franche-Comt\'{e}\\
25030 Besan\c{c}on cedex\\
France}

\email{jean.roydor@math.univ-fcomte.fr} \subjclass{47L30,47L25}

\maketitle
\begin{abstract} We give an operator space characterization of
subalgebras of $C(\Omega,M_n)$. We also describe injective
subspaces of $C(\Omega,M_n)$ and then give applications to
sub-TROs of $C(\Omega,M_n)$. Finally, we prove an `$n$-minimal
version' of the Christensen-Effros-Sinclair representation
theorem.
\end{abstract}

%\date{}
\maketitle

%%%%%%%%%%%%%%%%%%%%%%%%%%%%%%%%%%%%%%%%%%%
\section{Introduction and preliminaries}
Let $n \in \mathbb{N}^*$. An operator space $X$ is called
$n$-$minimal$ if there exists a compact Hausdorf space $\Omega$
and a completely isometric map $i:X \to C(\Omega,M_n)$. The
readers are referred to \cite{Pi1} and \cite{ER} for details on
operator space theory. Recall that the $C^*$-algebra
$C(\Omega,M_n)$ can be identified $*$-isomor\-phi\-cally with
$C(\Omega) \otimes_{min} M_n$ or $M_n(C(\Omega))$ (see
\cite[Proposition 12.5]{P} for details). Obviously, in the case
$n=1$, we just deal with the well-known class of minimal operator
spaces. Smith noticed that any linear map into $M_n$ is completely
bounded and its cb norm is achieved at the $n^{th}$ amplification
i.e. $\Vert u \Vert_{cb}=\Vert id_{M_n} \otimes u \Vert$ (see
\cite[Proposition 8.11]{P}). Clearly, this property remains true
for maps into $C(\Omega,M_n)$. In fact, Pisier showed that this
property characterized $n$-minimal operator spaces. More
precisely, if $X$ is an operator space such that any linear map
$u$ into $X$ is necessarily completely bounded and $\Vert u
\Vert_{cb}=\Vert id_{M_n} \otimes u \Vert$, then $X$
is $n$-minimal (see \cite[Theorem 18]{Pi}).  \\
\indent We now recall a few facts about injectivity (see
\cite{ER}, \cite{P} or \cite{BLM1} for details). A Banach space
$X$ is \textit{injective} if for any Banach spaces $Y \subset Z$,
each contractive map $u:Y \to X$ has a contractive extension
$\tilde{u}:Z \to X$. Since the 50's, it is known that a Banach
space is injective if and only if it is isometric to a
$C(K)$-space with $K$ a Stonean space and dual injective Banach
spaces are exactly $L^\infty$-spaces (see \cite{EOR} for more
details). More recently, injectivity has also been studied in
operator spaces category. Analogously, an operator space $X$ is
said to be \textit{injective} if for any operator spaces $Y
\subset Z$, each completely contractive map $u:Y \to X$ has a
completely contractive extension $\tilde{u}:Z \to X$. Note that a
Banach space is injective if and only if it is injective as a
minimal operator space. Let $X$ be an operator space, $(Y,i)$ is
an \textit{injective envelope of $X$} if $Y$ is an injective
operator space, $i:X \to Y$ is a complete isometry and for any
injective operator space $Z$ with $i(X) \subset Z \subset Y$, then
$Z=Y$. Sometimes, we may forget the completely isometric
embedding. In fact, any operator space admits a unique injective
envelope (up to complete isometry) and we write $I(X)$ the
injective envelope of
$X$. See \cite[Chapter 6]{ER} for a proof of this construction.\\
\indent Obviously, an $\ell^\infty$-direct sum of $n$-minimal
operator spaces is again $n$-minimal. In the next proposition, we
give some other easy properties of $n$-minimal operator spaces :

\begin{proposition}\label{1.1} Let $X$ be an $n$-minimal operator
space.  \begin{enumerate}[i)] \item Then its bidual $X^{**}$ and
its injective envelope $I(X)$ are $n$-minimal too. \item If
moreover, $X$ is a dual operator space, then there is a set $I$
and a $w^*$-continuous complete isometry $i:X \to
\ell_I^\infty(M_n)$.
\end{enumerate}
\end{proposition}
\begin{proof} The first assertion of $i)$ follows from
$C(\Omega,M_n)^{**}=M_n(C(\Omega))^{**}=M_n(C(\Omega)^{**})$
$*$-isomorphically. For the second, suppose $X \subset
C(\Omega,M_n)$ completely isometrically. From the description of
injective Banach spaces, $I(C(\Omega))=C(\Omega ')$ with $\Omega
'$ Stonean. Then $X \subset C(\Omega ',M_n)$ and this last
$C^*$-algebra is injective, so $I(X) \subset C(\Omega ',M_n)$
completely isometrically.\\
Suppose that $W$ is an operator space predual of $X$. Then
$X=CB(W,\mathbb{C})$ and if $I=\cup_n Ball(M_n(W))$, we have a
$w^*$-continuous complete isometry $\psi:X \longrightarrow
\oplus^\infty_{w \in I} M_{n_w}$ (where $n_w=m$ if $w \in M_m(W)$)
defined by $\psi(x)=( [ x(w_{ij}) ] )_{w \in I}$. Let $x \in
M_k(X)=CB(W,M_k)$. As $X$ is $n$-minimal, by \cite[Proposition
8.11]{P}, $\Vert x^* \Vert_{cb}=\Vert id_{M_n} \otimes x^* \Vert$,
where $x^*:M_k^* \to X$ denotes the adjoint map. However, for any
$l$, $\Vert id_{M_l} \otimes x\Vert = \Vert id_{M_l} \otimes x^*
\Vert$. Hence, $\Vert x \Vert_{cb}=\Vert id_{M_n} \otimes x \Vert$
and so, in the definition of $\psi$, we can majorize the $n_w$'s
by $n$ and obtain a complete isometry.
\end{proof}

We reviewed that an injective minimal operator space is a
$C^*$-algebra, but this property is lost for $n$-minimal operator
spaces (as soon as $n \geq 2$). Generally, an injective operator
space only admits a structure of ternary ring of operators. We
recall that a closed subspace $X$ of a $C^*$-algebra is a
\textit{ternary ring of operators} (TRO in short) if $X X^\star X
\subset X$, here $X^\star$ denotes the adjoint space of $X$. And a
$W^*$-\textit{TRO} is $w^*$-closed subspace of a von Neumann
algebra stable under the preceding `triple product'. TROs and
$W^*$-TROs can be regarded as generalization of $C^*$-algebras and
$W^*$-algebras. For instance, The Kaplansky density Theorem and
the Sakai Theorem remain valid for TROs (see e.g. \cite{EOR}). A
$triple~ morphism$ between TROs is a linear map which preserves
their `triple products'. This category enjoys some `rigidity
properties' like $C^*$-algebras category (see e.g. \cite{EOR} or
\cite[Section 8.3]{BLM1} for
details).\\
\indent So far we have seen that certain properties of the minimal
case `pass' to the $n$-minimal situation. Therefore, the basic
idea of this paper is to extend valid results in the commutative
case to the more general $n$-minimal case.
\\\\
\indent A first commutative result that can be extended to the
$n$-minimal case is a theorem on operator algebras due to Blecher.
We recall that an \textit{operator algebra} is a closed subalgebra
of $B(H)$, see \cite{BLM1} or \cite{P} for some backgrounds and
developments. And an operator algebra is said to be
\textit{approximately unital} if it possesses a contractive
approximate identity. In \cite{B}, Blecher showed that an
approximately unital operator algebra which is minimal is in fact
a uniform algebra (i.e a subalgebra of a commutative
$C^*$-algebra). So here, let $A$ be an approximately unital
operator algebra and assume that $A$ is $n$-minimal. Then we can
obtain a completely isometric homomorphism from $A$ into a certain
$C(\Omega,M_n)$ (see Corollary \ref{2.4}). Of course, we can ask
this type of question in various categories of operator spaces.
More precisely, let $\mathcal{C}$ denote a certain subcategory of
the category of operator spaces with completely contractive maps.
Let $X$ be an object of $\mathcal{C}$ which is $n$-minimal (as an
operator space), can we obtain a completely isometric morphism of
$\mathcal{C}$ from $X$ into a $C^*$-algebra of the form
$C(\Omega,M_n)$ ? For example in Proposition \ref{1.1}, we
answered this question in the category of dual operator spaces and
$w^*$-continuous completely contractive maps. We will also give
a positive answer in the category of : \\
- $C^*$-algebras and $*$-homomorphisms (see Theorem \ref{2.2}) ;\\
- von Neumann algebras and $w^*$-continuous $*$-homomorphisms (see Remark \ref{2.5}) ;\\
 - approximately unital operator algebras and completely
contractive
homomorphisms (see Corollary \ref{2.4}) ;\\
- operator systems and completely positive unital maps (see Corollary \ref{3.3}) ;\\
- TRO and triple morphisms (see Proposition \ref{4.1}) ;\\
- $W^*$-TRO and $w^*$-continuous triple morphisms (see Corollary \ref{4.5}).\\
It means that, in any of the previous categories, the $n$-minimal
operator space structure encodes the additional structure. Since
the injective envelope of an $n$-minimal operator space is
$n$-minimal too (see Proposition \ref{1.1}), passing to the
injective envelope will be a useful technique to answer these
preceding questions. In any case, the description of $n$-minimal
injective operator spaces (established in Theorem \ref{3.5})
will be of major importance.\\
\\
\indent The Christensen-Effros-Sinclair theorem (CES-theorem in
short) is a second example of theorem that could be treated in the
$n$-minimal case. Let $A$ be an operator algebra (or more
generally a Banach algebra endowed with an operator space
structure) and let $X$ be an operator space which is a left
$A$-module. Then following \cite[Chapter 3]{BLM1}, we say that $X$
is a left $h$-$module~over~A$ if the action of $A$ on $X$ induces
a completely contractive map from $ A \otimes_h X$ in $X$ (where
$\otimes_h$ denotes the Haagerup tensor product). The CES-theorem
states that if $X$ is a non-degenerate h-module over an
approximately unital operator algebra $A$ (i.e. $AX$ is dense in
$X$), then there exists a $C^*$-algebra $C$, a complete isometry
$i:X \to C$ and a completely contractive homomorphism $\pi:A \to
C$ such that $i(a \cdot x)=\pi(a)i(x)$ for any $a \in A,~x \in X$.
We will prove that if $X$ is $n$-minimal, we can choose $C$ to be
$n$-minimal too. This leads to an `$n$-minimal version' of the
CES-theorem. The case $n=1$ has been treated (see \cite{BLM2}) in
a Banach space framework ; here we will use an operator space
approach based on the multiplier algebra of an operator space.
\section{Subalgebras of $C(\Omega,M_n)$}
Recall that a $C^*$-algebra is \textit{subhomogeneous of degree}
$\leq n$ if it is contained $*$-isomorphically in a $C^*$-algebra
of the form $C(\Omega,M_n)$, where $\Omega$ is compact Hausdorf
space. Hence $n$-minimality could be seen as an operator space
analog of subhomogeneity of degree $\leq n$. We also recall the
well-known characterization of subhomogeneous $C^*$-algebras in
terms of representations. Indeed, a $C^{\ast}$-algebra $A$ is
subhomogeneous of degree $\leq n$ if and only if every irreducible
representation of $A$ has dimension no greater than $n$. The `if
part' is easily obtained taking a separating family of irreducible
representations. Conversely, if $A$ is contained
$*$-isomorphically in $C(\Omega,M_n)$, then every irreducible
representation of $A$ extends to one on $C(\Omega,M_n)$ (because
irreducible representations correspond to pure states). And as any
irreducible representation of $C(\Omega,M_n)$ has dimension no
greater than $n$, we can conclude (the author thanks Roger Smith
for these explanations).

 \begin{lemma}\label{2.1} Let $k \in \mathbb{N}^*$, $\Omega$ a compact Hausdorf space and $t_k$ the transpose mapping
 $$\begin{array}{ccll}
               t_k ~:&  C(\Omega,M_k)&\to &C(\Omega,M_k),\\
               &[f_{ij}]  & \mapsto & [f_{ji}]
               \end{array}$$ Then for any $l \in \mathbb{N}^*$, $\Vert id_{M_l} \otimes t_k \Vert =inf ( k,l )$.
               Thus $t_k$ is completely bounded and $\Vert id_{M_k} \otimes t_k
\Vert=\Vert t_k \Vert_{cb}=k$.
\end{lemma}
\begin{proof} The equality $\Vert t_k \Vert _{cb} = k$ is obtained in adapting
 the proof of \cite[Proposition 2.2.7]{ER}. Hence in the case $k \leq l$, by \cite[Proposition 8.11]{P})
 we obtain $\Vert id_{M_l} \otimes t_k \Vert =inf ( k,l )$. Next we prove $\Vert id_{M_l} \otimes t_k \Vert \leq l$.  let $\pi$ be the cyclical
permutation matrix $$\pi= \left(
\begin{array}{ccccc}
 0&0&\cdots & 0&I_k \\
 I_k&0 & \cdots&0&0\\
 \vdots&&\ddots \quad &&\vdots\\
 0&0&\cdots&I_k&0\\
 \end{array} \right) \in M_l(C(\Omega,M_k)).$$ Let $D_l:M_l(C(\Omega,M_k)) \to M_l(C(\Omega,M_k))$ be the
 diagonal truncation of $M_l$ i.e. $D_l(\epsilon_{ij} \otimes
 y)=\delta_{ij} \epsilon_{ij} \otimes
 y$ where $\epsilon_{ij}$ ($i,j \leq l$) denotes the matrix units
 of $M_l$ and $y \in C(\Omega,M_k)$. Let $x=[x_{ij}]_{i,j \leq
l} \in M_l(C(\Omega,M_k))$ and for simplicity of notation, we
wrote $t(x)=id_{M_l} \otimes
 t_k(x)\in M_l(C(\Omega,M_k))$. Then
 $t(x)=\sum_{i=0}^{l-1} D_l(t(x)\pi ^i)\pi ^{-i}$, and so
 $\Vert t(x) \Vert \leq \sum_{i=0}^{l-1} \Vert D_l(t(x)\pi ^i)\Vert $ (because $\pi$ is unitary).
 To conclude it suffices to majorize each terms of the previous sum by
 the norm of $x$.
However, for any $i$, $D_l(t(x)\pi ^i)$ is of the form
$\sum_{j=1}^{l} \epsilon_{jj} \otimes t_k(x_{p_jq_j})$ and we can
majorize its norm, $$\Vert \sum_{j=1}^{l} \epsilon_{jj} \otimes
t_k(x_{p_jq_j}) \Vert^2=\Vert \sum_{j=1}^{l} \epsilon_{jj} \otimes
t_k(x_{p_jq_j}x_{p_jq_j}^*) \Vert = max_j \{ \Vert
t_k(x_{p_jq_j}x_{p_jq_j}^*) \Vert \}$$ but
$x_{p_jq_j}x_{p_jq_j}^*$ is a selfadjoint element of
$C(\Omega,M_k)$, so its norm is unchanged by $t_k$ and $\Vert
t_k(x_{p_jq_j}x_{p_jq_j}^*) \Vert = \Vert x_{p_jq_j} \Vert ^2 \leq
\Vert x \Vert ^2$. Finally, for any $i$, $\Vert D_l(t(x)\pi ^i)
\Vert \leq \Vert x \Vert$ which enable us to conclude.\\
Moreover in adapting \cite[Proposition 2.2.7]{ER}, we have easily
$\Vert id_{M_l} \otimes t_k \Vert =l$, if $l \leq k$.  %Here, let $x \in C(\Omega,M_k)$
%and $$\pi= \left(
%\begin{array}{ccccc}
% 0&0&\cdots & 0&1 \\
 %1&0 & \cdots&0&0\\
 %\vdots&&\ddots \quad &&\vdots\\
 %0&0&\cdots&1&0\\
 %\end{array} \right) \in M_k(C(\Omega))$$ then $x=\sum_{i=0}^{k-1} D_k(x\pi ^i)\pi ^{-i}$ and
 %$t_k(x)=\sum_{i=0}^{k-1} \pi ^{i}D_k(x\pi
 %^i)$. Therefore, $$ \Vert t_k \Vert_{cb}=\sum_{i=0}^{k-1} \Vert x \mapsto \pi ^{i}D_k(x\pi ^i)
 %\Vert_{cb}
 %\leq k$$
 %Hence $\Vert id_{M_l} \otimes t_k \Vert =inf ( k,l )$
 %and the last assertion of the lemma follows.
\end{proof}

In the next theorem, we denote by $A^{op}$ the opposite structure
of a $C^*$-algebra $A$ (see e.g. \cite[Paragraph 2.10]{Pi1} or
\cite[Paragraph 1.2.25]{BLM1} for details). More generally, if $X$
is an operator space, $X^{op}$ is the same vector space but with
the new matrix norms defined by $$\Vert [x_{ij}] \Vert
_{M_n(X^{op})}=\Vert [x_{ji}] \Vert _{M_n(X)} \quad \mbox{for~any}
~[x_{ij}] \in M_n(X).$$ Hence the assumption $(iii)$ in the next
theorem is equivalent to $$\Vert id_A \otimes t_k \Vert \leq n
\quad \mbox{for~any} ~k \in \mathbb{N}^*,$$ where $t_k$ denotes
the transpose mapping from $M_k$ to $M_k$ discussed above.

\begin{theorem}\label{2.2} Let $A$ be a $C^*$-algebra. Then the following are equivalent :
\begin{enumerate}[(i)]
\item $A$ is subhomogeneous of degree $\leq n$. \item $A$ is
$n$-minimal. \item $\Vert id : A \to A^{op} \Vert_{cb} \leq n$.
\end{enumerate}
\end{theorem}
\begin{proof} $(i) \Rightarrow (ii)$ is obvious and $(ii) \Rightarrow
(iii)$ follows from the first equality in the previous lemma.
Suppose $(iii)$. Let $\pi : A \to B(H)$ be an irreducible
representation and $k \in \mathbb{N}^*$ such that $M_k \subset
B(H)$ ; from the first paragraph of this section, we must prove
that $k \leq n$. Using the previous lemma (with a singleton as
$\Omega$), there is $x \in M_k(M_k) \subset M_k(B(H))$ satisfying
$$k=\Vert id_{M_k} \otimes t_k(x) \Vert \quad \mbox{and}  \quad \Vert x \Vert \leq
1.$$ The representation $\pi_k=id_{M_k} \otimes \pi$ is also
irreducible so the commutant $\pi_k(M_k(A))'=\mathbb{C}I_{H^k}$,
thus by the von Neumann's double commutant theorem
$$\overline{M_k(\pi(A))}^{so}=M_k(B(H)).$$  Then by the Kaplansky density theorem, there exists a net
$(x_\lambda)_{\lambda \in \Lambda} \subset M_k(\pi(A))$ converging
to $x$ in the $\sigma$-strong operator topology and such that
$\Vert x_\lambda \Vert \leq 1$. Therefore $id_{B(H)} \otimes
t_k(x_\lambda)$ tends to $id_{M_k} \otimes t_k(x)$ in the
$w^*$-topology and by the semicontinuity of the norm in the
$w^*$-topology, we have
$$k=\Vert id_{M_k} \otimes t_k(x) \Vert \leq \limsup_\lambda \Vert
id_{B(H)} \otimes t_k(x_\lambda) \Vert$$ Let $\epsilon
> 0$. For any $\lambda$, there exists $y_\lambda \in M_k(A)$
such that $x_\lambda=\pi_k(y_\lambda)$ and $\Vert y_\lambda \Vert
\leq 1+\epsilon$. By assumption, $$\Vert id_A \otimes t_k \Vert
\leq n$$ Moreover $(id_{B(H)} \otimes t_k) \circ \pi_k=\pi_k \circ
(id_A \otimes t_k)$. Combining these arguments we finally obtain
$$\begin{array}{ccl}
               k=\Vert id_{M_k} \otimes t_k(x) \Vert&\leq&\limsup_\lambda \Vert
id_{B(H)} \otimes t_k(\pi_k(y_\lambda)) \Vert\\
               &\leq&\limsup_\lambda \Vert
\pi_k(id_{A} \otimes t_k(y_\lambda)) \Vert\\
&\leq& \Vert id_{A} \otimes t_k \Vert(1+ \epsilon)\\
&\leq& n(1+ \epsilon).\\
               \end{array}$$ Hence $k \leq n$.
\end{proof}

Now we extend $(i) \Leftrightarrow (ii)$ of the previous theorem,
which concerns $C^*$-algebras, to the larger category of operator
algebras and completely contractive homomorphisms.

\begin{corollary} \label{2.4}Let $A$ be an approximately unital
operator algebra. Then the following are equivalent :
\begin{enumerate}[(i)]
\item There exists a compact Hausdorf space $\Omega$ and a
completely isometric homomorphism $\pi:A \to C(\Omega,M_n)$. \item
$A$ is $n$-minimal.
\end{enumerate}
\end{corollary}
\begin{proof} $(i) \Rightarrow (ii)$ is obvious. Suppose $(ii)$. We know that the injective envelope $I(A)$
is a $C^*$-algebra and there is a completely isometric
homomorphism from $A$ into $I(A)$ (see \cite[Corollary
4.2.8]{BLM1}). Since $A$ is $n$-minimal, $I(A)$ is $n$-minimal
too, by Proposition \ref{1.1}. Applying Theorem \ref{2.2} to
$I(A)$, we can conclude.
\end{proof}

 \begin{remark}\label{2.5} Using the well-known description of subhomogeneous $W^*$-algebras,
  we easily obtained that, if $M$ is a $W^*$-algebra and $M$ is $n$-minimal, then
  $$M=\oplus^{\infty}_{i \in I} L^\infty(\Omega_i,M_{n_i})$$ via a normal $*$-isomorphism.
  Here $\Omega_i$ is a measure space and
  $n_i \leq n$, for any $i \in I$. This result will be extended to the
 category of $W^*$-TROs (see Corollary \ref{4.5}).
 \end{remark}

\section{Injective $n$-minimal operator spaces}
Before describing injective $n$-minimal operator spaces, we can
treat the more `rigid' case of injective $n$-minimal
$C^*$-algebras as an easy consequence of \cite{SW}.

\begin{proposition}\label{3.1} Let $A$ be an $n$-minimal $C^*$-algebra. Then the following are equivalent :
\begin{enumerate}[(i)]
\item $A$ is injective. \item There exists a finite family of
Stonean compact Hausdorf spaces $(\Omega_i)_{i \in I}$ such that
$A=\oplus^{\infty}_{i \in I} C(\Omega_i,M_{n_i})$
$*$-isomorphically with $n_i \leq n$, for any $i \in I$.
\end{enumerate}
\end{proposition}
\begin{proof} As $A$ is injective, $A$ is monotone complete (see \cite[Theorem 6.1.3]{ER}). Thus $A$
is an $AW^*$-algebra. Moreover, by \cite[Proposition 6.6]{SW}, $A$
either contains $M_\infty=\oplus^{\infty}_k M_k$ or $A$ is of the
desired form. The first alternative is impossible because $A$ is
$n$-minimal, which ends the `only if' part. The converse is clear,
since each $\Omega_i$ is Stonean.
\end{proof}

\begin{remark}\label{3.2} This theorem enables us to give a short
proof of $(ii) \Rightarrow (i)$ in Theorem \ref{2.2}. If $A$ is an
$n$-minimal $C^*$-algebra, its injective envelope $I(A)$ is
$n$-minimal too (by Proposition \ref{1.1}). $I(A)$ is a
$C^*$-algebra and contains $A$ $*$-isomorphically (see
\cite[Theorem 6.2.4]{ER}). Applying the previous proposition to
$I(A)$, we obtain that
$$I(A)=\oplus^{\infty}_{i \in I} C(\Omega_i,M_{m_i}) \quad \mbox{
$*$-isomorphically}$$ with $n_i \leq n$, for any $i \in I$. And
now it is not difficult to construct a $*$-isomorphism from $A$
into $C(\Omega,M_n)$ where $\Omega$ denotes the (finite) disjoint
union of the $\Omega_i$'s.
\end{remark}

We recall that an operator space $X$ is \textit{unital} if there
exists $e \in X$ and a complete isometry from $X$ into a certain
$B(H)$ which sends $e$ on $I_H$. From the result below, an
$n$-minimal operator system can embed into a $C^*$-algebra of the
form $C(\Omega,M_n)$ via a unital complete order isomorphism.

\begin{corollary}\label{3.3} Let $X$ be a unital operator space. Then the following are equivalent :
\begin{enumerate}[(i)]
\item There exists a compact Hausdorf space $\Omega$ and a
completely isometric unital map $\pi:X \to C(\Omega,M_n)$. \item
$X$ is $n$-minimal.
\end{enumerate}
\end{corollary}
\begin{proof} $(i) \Rightarrow (ii)$ is obvious. Suppose $(ii)$. We know that the injective envelope $I(X)$
is a $C^*$-algebra and there is a unital complete isometry from
$X$ into $I(X)$ (see \cite[Corollary 4.2.8]{BLM1}). As $X$ is
$n$-minimal, $I(X)$ is $n$-minimal too (by Proposition \ref{1.1}).
By the previous theorem
$$I(X)=\oplus^{\infty}_{i \in I} C(\Omega_i,M_{n_i}) \mbox{
$*$-isomorphically}.$$ Next we show that for any $i$ there exists
a unital complete isometry $\varphi_i:M_{n_i} \to M_n$. By
iteration, we only need to prove that for any $k \in
\mathbb{N}^*$, there exists a unital complete isometry from $M_k$
into $M_{k+1}$. The map
$$\begin{array}{ccll}
               i_k ~:& M_k&\to &M_{k+1}\\
               &x & \mapsto & x \oplus tr_k(x)
               \end{array}$$
               (where $tr_k$ denotes the normalized trace on $M_k$) is
               a unital complete order isomorphism and thus a
               unital complete isometry. We can define a
               unital complete isometry $$\begin{array}{ccll}
               \psi ~:& \oplus^{\infty}_{i \in I} C(\Omega_i,M_{n_i})&\to &C(\Omega,M_n)\\
               &(f_i \otimes x_i )_i & \mapsto & \sum_i \tilde{f}_i \otimes
               \varphi_i(x_i)
               \end{array}$$ where $\Omega$ denotes the disjoint
               union of $\Omega_i$'s and $\tilde{f}_i$ the
               continuous extension by 0 of $f_i$ on $\Omega$.
               Finally, we have $$X \subset I(X) \subset C(\Omega,M_n)$$ via
               unital complete isometries.
\end{proof}

\begin{remark}\label{3.4} This last corollary cannot be extended to the
category of operator algebras and completely contractive
homomorphisms. In fact, if $\pi:M_p \to C(\Omega,M_q)$ is a unital
completely contractive homomorphism then $\pi$ is positive so it
is a $*$-homomorphism. Therefore (composing by an evaluation) we
can obtain a unital $*$-homomorphism from $M_p$ in $M_q$ and thus
$p$ divides $q$ (see \cite[Exercise 4.11]{P}).
\end{remark}

We must recall a crucial construction of the injective envelope of
an operator space $X$ which will be useful in this paper (see
 \cite[Paragraph 4.4.2]{BLM1} for more details on this
 construction). Assume that $X \subset B(H)$, we can consider its Paulsen system
$$S(X)=\left( \begin{array}{cc}
 \mathbb{C} & X \\
 X^\star & \mathbb{C}
 \end{array} \right) \subset M_2(B(H))$$ where $X^\star$ denotes the adjoint space of $X$.
 The injective envelope of $S(X)$
 is the range of a completely contractive projection $\varphi:M_2(B(H)) \to M_2(B(H))$ which leaves $S(X)$ invariant.
 By \cite[Theorem 6.1.3]{ER}, $I(S(X))$ admits a $C^*$-algebraic structure but it is not necessarily a sub-$C^*$-algebra of
 $M_2(B(H))$.
 However
 $$p=\left( \begin{array}{cc}
 1 & 0 \\
 0 & 0
 \end{array} \right) \quad \mbox{and} \quad q=\left( \begin{array}{cc}
 0 & 0 \\
 0 & 1
 \end{array} \right)=1-p$$ (which are invariant by $\varphi$) are
 still orthogonal projections (i.e. selfadjoint idempotents) of the new $C^*$-algebra
 $I(S(X))$. Since they satisfy $p+q=1$ and $pq=0$, we can decompose $I(S(X))$ in $2 \times 2$ matrices, as follow
 : $$I(S(X))=\left( \begin{array}{cc}
 I_{11}(X) & I_{12}(X) \\
 I_{21}(X) & I_{22}(X)
 \end{array} \right) $$ where $I_{11}(X)=pI(S(X))p$ and
 $I_{22}(X)=qI(S(X))q$ are injective $C^*$-algebras,
 $I_{12}(X)=pI(S(X))q$ is in fact the injective envelope of $X$ and
 $I_{21}(X)=qI(S(X))p$ coincides with $I_{12}(X)^\star$. Therefore, we obtain the
 Hamana-Ruan Theorem i.e. an injective operator space
 is an `off-diagonal' corner of an injective $C^*$-algebra (see \cite[Theorem 6.1.6]{ER}). It links the study of injective
 operator spaces to injective $C^*$-algebras (and, by the
 way, it proves that an injective operator space is a TRO).

\begin{theorem}\label{3.5} Let $X$ be an $n$-minimal operator space. Then the following are equivalent :
\begin{enumerate}[(i)]
\item $X$ is injective. \item There exists a finite family of
Stonean compact Hausdorf spaces $(\Omega_i)_{i \in I}$ such that
$X=\oplus^{\infty}_{i \in I} C(\Omega_i,M_{r_i,k_i})$ completely
isometrically with $r_i,k_i \leq n$, for any $i \in I$.
\end{enumerate}
\end{theorem}
\begin{proof}$(ii) \Rightarrow (i)$ is obvious. Let $X$ be an injective
$n$-minimal operator space. By the discussion above, we know that
there exists an injective $C^*$-algebra $A$ and a projection $p
\in A$ such that
$$X=pA(1-p) \quad \mbox{completely
isometrically}$$ In fact $A$ is the injective envelope of $S(X)$
the Paulsen system of $X$ (see above). As $X$ is $n$-minimal,
$S(X)$ is $2n$-minimal, so is $A$ (by Proposition \ref{1.1}). From
Proposition \ref{3.1},
$$A=\oplus^{\infty}_{i \in I} C(\Omega_i,M_{m_i}) \quad \mbox{
$*$-isomorphically}$$ where $m_i \leq 2n$. For simplicity of
notation, we will assume momentarily that the cardinal of $I$ is
equal to 1 and so
$$X=pC(\Omega,M_m)(1-p) \quad \mbox{completely isometrically},$$
for some projection $p\in C(\Omega,M_m)$. Using \cite[Corollary
3.3]{DP} or \cite[Theorem 3.2]{GP}, there is a unitary $u$ of
$C(\Omega,M_m)$ such that for any $\omega \in \Omega$,
$upu^*(\omega)$ is of the form $diag(1,\dots,1,0,\dots,0)$. So we
may assume that for any $\omega \in \Omega$, $p(\omega)$ is a
diagonal matrix of the form given above. For any $k \leq m$, we
define
$$\Omega_k=\{ \omega \in \Omega~:~rg(p(\omega))=k \}$$ which is a
closed subset of $\Omega$ (because the rank and the trace of a
projection coincide) and the family $(\Omega_k)_{k \leq m}$ forms
a partition of $\Omega$. Hence, any $\Omega_k$ is open (and
closed) in $\Omega$, so $\Omega_k$ is still Stonean. We have the
completely isometric identifications
$$X=pC(\Omega,M_m)(1-p)=\oplus^{\infty}_{ k \leq m}
C(\Omega_k,M_{k,m-k})=\oplus^{\infty}_{1 \leq k \leq m-1}
C(\Omega_k,M_{k,m-k}).$$ Moreover, for any $1 \leq k \leq m-1$, we
have the completely isometric embeddings
$$M_{k,m-k} \subset C(\Omega_k,M_{k,m-k}) \subset X$$ and as $X$
is $n$-minimal, it forces $k \leq n$ and $m-k \leq n$ ; if not, at
least the row Hilbert space $R_{n+1}$ or the column Hilbert space
$C_{n+1}$ would be $n$-minimal. Thus $X$ has the announced form.
In general, $I$ is a finite set and $$X=p \oplus^{\infty}_{i \in
I} C(\Omega_i,M_{m_i}) (1-p)=\oplus^{\infty}_{i \in I} p_i
C(\Omega_i,M_{m_i})(1-p_i)$$ where $p_i$ is a projection in
$C(\Omega_i,M_{m_i})$ and $p=\oplus_i p_i$. Applying the preceding
argument to each terms $p_iC(\Omega_i,M_{m_i})(1-p_i)$, we can
conclude.
\end{proof}

\begin{corollary}\label{3.7} Let $X$ be an $n$-minimal dual operator space. Then the following are equivalent :
\begin{enumerate}[(i)]
\item $X$ is injective. \item There exists a finite
family of measure spaces $(\Omega_i)_{i \in I}$ such that \\
$X=\oplus^{\infty}_{i \in I} L^\infty(\Omega_i,M_{r_i,k_i})$ via a
completely isometric $w^*$-homeomorphism with $r_i,k_i \leq n$,
for any $i \in I$.
\end{enumerate}
\end{corollary}
\begin{proof} From the previous theorem, $X=\oplus^{\infty}_{i}
C(K_i,M_{r_i,k_i})$ completely isometrically, where $K_i$ is
Stonean. Since $X$ is a dual operator space, it forces $C(K_i)$ to
be a dual commutative $C^*$-algebra i.e.
$C(K_i)=L^\infty(\Omega_i)$ (via a normal $*$-isomorphism) for
some measure space $\Omega_i$.
\end{proof}

\section{Application to $n$-minimal TROs}
In this section, we will use the description of injective
$n$-minimal operator spaces to obtain results on $n$-minimal TROs.
First, we will see that the $n$-minimal operator structure of a
TRO determines its whole triple structure. See e.g. \cite{EOR} or
\cite[Section 8.3]{BLM1} for details on TROs.

\begin{proposition}\label{4.1} Let $X$ be a TRO. The following are equivalent :
\begin{enumerate}[(i)]
\item There exists a compact Hausdorf space $\Omega$ and an
injective triple morphism $\pi:X \to C(\Omega,M_n)$. \item $X$ is
$n$-minimal.
\end{enumerate}
\end{proposition}
\begin{proof}$(i) \Rightarrow (ii)$ follows from the fact that an
injective triple morphism is necessarily completely isometric (see
e.g. \cite[Proposition 2.2]{EOR}
or \cite[Lemma 8.3.2]{BLM1}).\\
Suppose $(ii)$. By \cite[Remark 4.4.5 (1)]{BLM1}, the injective
envelope of $X$ admits a TRO structure and $X$ can be viewed
 as a
sub-TRO of $I(X)$. From Theorem \ref{3.5}, we can describe this
injective envelope as a direct sum,
$$I(X)=\oplus^{\infty}_{i \in I} C(\Omega_i,M_{r_i,k_i}) \quad \mbox{completely
isometrically}.$$ But the right hand side of the equality admits a
canonical TRO structure and it is known (see e.g. \cite[Corollary
4.4.6]{BLM1}) that a surjective complete isometry between TROs is
automatically a triple morphism. In addition, for any $i$, the
embedding $\varphi_i:M_{r_i,k_i} \to M_n$ into the `up-left'
corner of $M_n$ is an injective triple morphism. As in the end of
the proof of Corollary \ref{3.3}, we finally obtain
$$X \subset I(X)=\oplus^{\infty}_{i \in I} C(\Omega_i,M_{r_i,k_i}) \subset
C(\Omega,M_n)$$ as TROs.
\end{proof}

For details on $C^*$-modules theory, the readers are referred to
\cite{L} or \cite[Chapter 8]{BLM1} for an operator space approach.
We must recall the construction of \textit{the linking
$C^*$-algebra} of a $C^*$-module. If $X$ is left $C^*$-module over
a $C^*$-algebra $A$ then its conjugate vector space $\overline{X}$
is a right $C^*$-module over $A$ with the action $\overline{x}
\cdot a=\overline{a^*x}$ and inner product $\langle \overline{x} ,
\overline{y} \rangle = \langle x , y \rangle$, for any $a \in A$,
$x,y \in X$. We denote by $_A \mathbb{K}(X)$ the $C^*$-algebra of
`compact' adjointable maps of $X$ and then
$$\mathcal{L}(X)=\left( \begin{array}{cc}
 A & X \\
 \overline{X} & _A \mathbb{K}(X)
 \end{array} \right)$$ is a $C^*$-algebra too which is called
 \textit{the linking $C^*$-algebra of} $X$. If $X$ is an equivalence bimodule (see
 \cite[Paragraph 8.1.2]{BLM1}) over two $C^*$-algebras $A$ and $B$, we define
 $$\mathcal{L}(X)=\left( \begin{array}{cc}
 A & X \\
 \overline{X} & B
 \end{array} \right) \quad \mbox{and} \quad \mathcal{L}^1(X)=\left( \begin{array}{cc}
 A^1 & X \\
 \overline{X} & B^1
 \end{array} \right)$$ (where $A^1$ and $B^1$ denote the unitizations of $A$ and $B$) which are also $C^*$-algebras (see \cite[Paragraph 8.1.17]{BLM1} for
 details on linking $C^*$-algebra). We can
 notice that $X$ is an `off-diagonal' corner of a $C^*$-algebra
 i.e. $X=p\mathcal{L}^1(X) (1-p)$ for some projection $p \in \mathcal{L}^1(X)$.
 Hence a $C^*$-module admits a TRO structure. The converse will be seen later on, which
 will make the correspondence between $C^*$-modules, equivalence bimodules and TROs (see \cite[Paragraph 8.1.19, 8.3.1]{BLM1}). Thus the
 next corollary is a reformulation of the previous
 proposition in the $C^*$-modules language. However, this corollary on representation of
 module action can be compared with Theorem \ref{5.3}.

\begin{corollary}\label{4.2} Let $X$ be a full left $C^*$-module over a $C^*$-algebra $A$. Then the following are equivalent :
\begin{enumerate}[(i)]
\item There exists a compact Hausdorf space $\Omega$, a complete
isometry $i:X \to C(\Omega,M_n)$ and a $*$-isomorphism $\sigma:A
\to C(\Omega,M_n)$ such that for any $a \in A$, $x,y \in X$
$$\begin{array}{c}
             i(a \cdot x)=\sigma(a)i(x)  \\
\sigma(\langle x,y \rangle)=i(x)i(y)^*\\
               \end{array}$$ \item $X$ is $n$-minimal and $A$ is subhomogeneous of degree $\leq n$.
                \item $X$ is $n$-minimal.
\end{enumerate}
\end{corollary}
\begin{proof} Only $(iii) \Rightarrow (i)$ needs a proof. Since $X$ is a $C^*$-module, it's also a TRO (see
above). From Proposition \ref{4.1}, there exists a compact
Hausdorf space $\Omega$ and an injective triple morphism $i:X \to
C(\Omega,M_n)$. By \cite[Corollary 8.3.5]{BLM1}, we can construct
a corner preserving $*$-isomorphism $\pi : \mathcal{L}(X) \to
M_2(C(\Omega,M_n))$ such that $i=\pi_{12}$. Choosing
$\sigma=\pi_{11}$, we obtain the desired relations.
\end{proof}

An equivalence bimodule version of the previous corollary could be
stated. In the previous result we transfer $n$-minimality from $X$
to $A$. We can treat the `reverse' question ; let $X$ be an
equivalence bimodule over two $n$-minimal $C^*$-algebras, we will
prove that $X$ is $n$-minimal. But first, let us translate this
proposition in the TROs language. Let $X$ be a TRO contained in a
$C^*$-algebra $B$ via an injective triple morphism. As in the
notation of the second section of \cite{R}, we define $C(X)$
(resp. $D(X)$) the norm closure of $span \{xy^*,~x,y \in X \}$
(resp. $span \{x^*y,~x,y \in X \}$). As $X$ is a sub-TRO of $B$,
$C(X)$ and $D(X)$ are sub-$C^*$-algebras of $B$ and
$$A(X)=\left( \begin{array}{cc}
 C(X) & X \\
 X^\star & D(X)
 \end{array} \right)$$ is a sub-$C^*$-algebras of $M_2(B)$. Hence
 a TRO can be regarded as an `off-diagonal' corner of a
 $C^*$-algebra which prove totally
 the correspondence between $C^*$-modules, equivalence bimodules and
 TROs. And $A(X)$ is also called \textit{the linking $C^*$-algebra
 of}
 $X$. Analogously, in $W^*$-TROs category, let $X$ be a $W^*$-TRO contained in a
$W^*$-algebra $B$ via a $w^*$-continuous injective triple
morphism. We define $M(X)$ (resp. $N(X)$) the $w^*$-closure of
$span \{xy^*,~x,y \in X \}$ (resp. $span \{x^*y,~x,y \in X \}$).
As $X$ is a sub-$W^*$-TRO of $B$, $M(X)$ and $N(X)$ are
sub-$W^*$-algebras of $B$ and
$$R(X)=\left( \begin{array}{cc}
 M(X) & X \\
 X^\star & N(X)
 \end{array} \right)$$ is a sub-$W^*$-algebras of $M_2(B)$. It is called \textit{the linking von Neumann algebra of}
 $X$. In fact, the linking algebras do not depend on the embedding of $X$ into a $C^*$-algebra.\\
 Obviously, if $X$ is an equivalence bimodule over two $C^*$-algebras $A$ and $B$,
 $C(X)$ and $D(X)$ play the roles of $A$ and $B$ in the correspondence between equivalence bimodules and
 TROs. Hence in the TROs language, we obtain (in the dual case) :

\begin{proposition}\label{4.3} Let $X$ be a $W^*$-TRO such that $M(X)$ and $N(X)$ are
$n$-minimal von Neumann algebras. Then $X$ is $n$-minimal and
$$X=\oplus^{\infty}_i ~ L^\infty(\Omega_i) \overline{\otimes} M_{r_i,k_i}$$
where $\Omega_i$ is a measure space, $r_i,k_i \leq n$, for any
$i$.
\end{proposition}

\begin{proof} We write $R(X)$ the linking von Neumann of $X$. From
\cite[Theorem 6.5.2]{KadR}, there exist $p_1,p_2$ and $p_3$ three
central projections of $R(X)$ such that $$R(X)= p_1R(X)
\oplus^{\infty} p_2R(X) \oplus^{\infty} p_3R(X)$$ and for
$i=1,2,3$, $p_iR(X)$ is a von Neumann algebra of type $i$ or
$p_i=0$. Since $M(X)$ is $n$-minimal, $M(X)$ is of type $I$.
However, $M(X)=pR(X)p$ for some projection $p$ in $R(X)$ and for
any $i$,
$$p_iM(X)=pp_ipM(X)pp_ip$$ As the type is unchanged by compression
(see \cite[Exercise 6.9.16]{KadR}), $p_iM(X)$ is of type $I$ or
$p_iM(X)=0$. On the other hand, for any $i$,
$$p_iM(X)=p_ipR(X)=pp_iR(X)p_ip$$ so $p_iM(X)$ has the same type as
$p_iR(X)$ or $p_iM(X)=0$. Thus $p_iM(X)=0$ for $i=2,3$ i.e.
$p_ip=0$ for $i=2,3$. Symmetrically, using our assumption on
$N(X)$, we have $p_i(1-p)=0$ for $i=2,3$. Hence $p_i=0$ for
$i=2,3$ i.e. $R(X)$ is of type $I$. Using \cite[Theorem 4.1]{R},
$$X=\oplus^{\infty}_k ~ L^\infty(\Omega_k) \overline{\otimes} M_{I_k,J_k}$$
where $\Omega_k$ is a measure space, $I_k,J_k$ are sets and
$M_{I_k,J_k}=B(\ell^2_{I_k},\ell^2_{J_k})$. Since $M(X)$ (resp.
$N(X)$) is $n$-minimal, it forces the cardinal of $I_k$ (resp.
$J_k$) to be no greater than $n$, for any $k$. So $X$ is
$n$-minimal and has the desired form.
\end{proof}

\begin{remark}\label{4.4} In the next two results, we will use that the multiplier
algebra of an $n$-minimal $C^*$-algebra is $n$-minimal too. It is
due to Proposition \ref{1.1}.
\end{remark}

The next corollary on $W^*$-TROs extends Remark \ref{2.5}.

\begin{corollary}\label{4.5} Let $X$ be a $W^*$-TRO. The following are equivalent :
\begin{enumerate}[(i)]
\item $X$ is $n$-minimal.  \item There exists a measure space
$\Omega$ and a $w^*$-continuous injective triple morphism $\pi:X
\to L^\infty(\Omega,M_n)$. \item There exists a finite
family of measure spaces $(\Omega_i)_{i \in I}$ such that \\
$X=\oplus^{\infty}_{i \in I} L^\infty(\Omega_i,M_{r_i,k_i})$ with
$r_i,k_i \leq n$, for any $i \in I$.
\end{enumerate}
\end{corollary}

\begin{proof} Only $(i) \Rightarrow (iii)$ needs a proof. Suppose
$(i)$. From Proposition \ref{4.1}, we can see $X$ as a sub-TRO of
$C(\Omega,M_n)$, hence by construction $C(X)$ and $D(X)$ are
$n$-minimal $C^*$-algebras. By \cite{KR}, $M(X)$ (resp. $N(X)$) is
the multiplier algebra of $C(X)$ (resp. $D(X)$), so $M(X)$ and
$N(X)$ are $n$-minimal $W^*$-algebras (by Remark \ref{4.4}). The
result follows from the previous proposition.
\end{proof}

Finally, we can generalize $(ii) \Leftrightarrow (iv)
\Leftrightarrow(v)$ of \cite[Proposition 8.6.5]{BLM1} on minimal
TROs to the $n$-minimal case.

\begin{theorem}\label{4.6} Let $X$ be a TRO, the following are equivalent :
\begin{enumerate}[(i)]
\item $X$ is $n$-minimal. \item $X^{**}$ is an injective
$n$-minimal operator space (see Corollary \ref{3.7}). \item $C(X)$
and $D(X)$ are $n$-minimal $C^*$-algebras.
\end{enumerate}
\end{theorem}
\begin{proof} $(ii) \Rightarrow (i)$ and $(i) \Rightarrow (iii)$ are
obvious. Suppose $(iii)$. From \cite[Proposition 2.4]{KR}, we know
that the multiplier algebra of $C(X^{**})$ is $C(X)^{**}$ and this
$C^*$-algebra is $n$-minimal by our assumption on $C(X)$ and
Remark \ref{4.4}. Moreover by \cite{R}, $M(X^{**})$ is also the
multiplier algebra of $C(X^{**})$, so $M(X^{**})$ is $n$-minimal
too. The same argument works for $N(X^{**})$ and we can apply
Proposition \ref{4.3} to $X^{**}$.
\end{proof}
\section{An $n$-minimal version of the CES-theorem}
To prove the `$n$-minimal' version the CES-Theorem we need the
notion of \textit{left multiplier algebra} of an operator space
$X$. A left multiplier of an operator space $X$ is a map $u:X \to
X$ such that there exist a $C^*$-algebra $A$ containing $X$ via a
complete isometry $i$ and $a \in A$ satisfying $i(u(x))=ai(x)$ for
any $x \in X$. Let $\mathcal{M}_l(X)$ denote the set of left
multipliers of $X$. And \textit{the multiplier norm of $u$} is the
infimum of $\Vert a \Vert$ over all possible $A,i,a$ as above. In
fact Blecher-Paulsen proved that any left multiplier can be
represented in the embedding of $X$ into the $C^*$-algebra
(discussed in section 3)  $$I(S(X))=\left(
\begin{array}{cc}
 I_{11}(X) & I(X) \\
 I(X)^\star & I_{22}(X)
 \end{array} \right) $$ More
precisely, for any left multiplier $u$ of norm no greater than 1,
there exists a unique $a \in I_{11}(X)$ of norm no greater than 1
such that $u(x)=ax$ for any $x \in X$ (see \cite[Theorem
4.5.2]{BLM1}). This result enables us to consider
$\mathcal{M}_l(X)$ as an operator subalgebra of $I_{11}(X)$ (see
the proof of \cite[Proposition 4.5.5]{BLM1} and \cite[Paragraph
4.5.3]{BLM1} for more details) and
$$\mathcal{M}_l(X)=\{ a \in I_{11}(X),~aX \subset X \}$$ as operator
algebras. The product used in the preceding centered formula is
the one on the $C^*$-algebra $I(S(X))$. And the operator algebra
$\mathcal{M}_l(X)$ is called \textit{the multiplier algebra of }
$X$. We let $\mathcal{A}_l(X)=\Delta(\mathcal{M}_l(X))$ denote the
diagonal (see \cite[Paragraph 2.1.2]{BLM1}) of $\mathcal{M}_l(X)$,
this $C^*$-algebra is called \textit{the left adjointable
multiplier algebra} of $X$ and
$$\mathcal{A}_l(X)=\{ a \in I_{11}(X),~aX \subset X~ \mbox{and} ~a^*X
\subset X \}$$ $*$-isomorphically. In fact, if $X$ happens to be
originally a $C^*$-algebra, $\mathcal{A}_l(X)$ is
just its multiplier algebra, and we recover Remark \ref{4.4}.\\
Symmetrically, \textit{the right multiplier algebra of X} is given
by
$$\mathcal{M}_r(X)=\{ b \in I_{22},~Xb \subset X \}$$ and
its diagonal $\mathcal{A}_r(X)=\{ b \in I_{22},~Xb \subset X
~\mbox{and} ~Xb^* \subset X \}$ is \textit{the right adjointable
multiplier algebra} of $X$.

\begin{lemma}\label{5.1} Let $X$ be an operator space and $I(X)$ its injective
envelope. Then there exists a completely contractive unital
homomorphism $\theta:\mathcal{M}_l(X) \to \mathcal{M}_l(I(X))$
such that $\theta(u)_{\vert X}=u$, for any $u \in
\mathcal{M}_l(X)$. And thus, $\theta _{\vert
\mathcal{A}_l(X)}:\mathcal{A}_l(X) \to \mathcal{A}_l(I(X))$
is a $*$-isomorphism.\\
Moreover, the same results hold for right multipliers.
\end{lemma}

\begin{proof} Let $u \in \mathcal{M}_l(X)$, then $u$ can be
represented by an element $a$ in $\{ a \in I_{11}(X),~aX \subset X
\}$. And using the multiplication inside $I(S(X))$, $aI(X) \subset
I(X)$, so $a$ can be seen as an element of $\mathcal{M}_l(I(X))$
which will be written $\theta(u)$. Therefore, $\theta$ is an
injective unital completely contractive homomorphism. The rest of
the proof follows from \cite[Paragraph 2.1.2]{BLM1}.
\end{proof}

In the next lemma, we use the \textit{$C^*$-envelope} of a unital
operator space, see \cite[Theorem 4.3.1]{BLM1} for details. And we
write $R_n$ (resp. $C_n$) the row (resp. column) Hilbert space of
dimension $n$. If $X$ is an operator space, we let $C_n(X)$ be the
minimal tensor product of $C_n$ and $X$ or equivalently
$$C_n(X)=\Big\{ \left( \begin{array}{cccc}
 x_1 & 0 & \cdots & 0 \\
 \vdots & \vdots & \cdots & \vdots\\
x_n &0 &\cdots &0
 \end{array} \right),~x_i \in X \Big\} ~ \subset
 M_{n}(X).$$ The definition of $R_n(X)$ is similar using a row instead of a column. Adapting the proof of the first example of the third
section of \cite{Z}, we can obtain :

\begin{lemma}\label{5.2} Let $A$ be an injective $C^*$-algebra and $k \in \mathbb{N}^*$. Then
\begin{enumerate}[(1)]
\item $\mathcal{M}_l(R_k(A))=A$ $*$-isomorphically and the action
is given by : $$a \cdot (x_1,\dots,x_k)=(ax_1,\dots,ax_k), \quad
\mbox{for~any}~ a, x_i \in A$$ \item $\mathcal{M}_r(C_k(A))=A$
$*$-isomorphically and the action is given by : $$\left(
\begin{array}{c} x_1 \\
                 \vdots\\
                 x_k \end{array} \right) \cdot a = \left(
\begin{array}{c} x_1a \\
                 \vdots\\
                 x_ka \end{array} \right), \quad \mbox{for~any}~ a, x_i \in A$$
\end{enumerate}

\end{lemma}
\begin{proof} We only prove $(1)$, the proof of $(2)$ is similar. Since
 $R_n=B(\ell_n^2,\mathbb{C})$, the Paulsen system $\mathcal{S}$ of
 $R_n(A)$ is
$$\mathcal{S}=\Big\{ \left( \begin{array}{cc}
 \alpha 1_A & x \\
 y* & \beta I_n \otimes 1_A
 \end{array} \right),~\alpha, \beta \in \mathbb{C},~x,y \in R_n(A) \Big\} ~ \subset
 M_{n+1}(A).$$
 Clearly the $C^*$-algebra $C^*(\mathcal{S})$ generated by $\mathcal{S}$ (inside $M_{n+1}(A)$)
 coincides with $M_{n+1}(A)$.
 Next we show that the $C^*$-envelope $C^*_e(\mathcal{S})$ of $\mathcal{S}$
 is $M_{n+1}(A)$. By the universal property of $C^*_e(\mathcal{S})$,
 there is a surjective $*$-homomorphism  $\pi: C^*(\mathcal{S})
 \twoheadrightarrow
 C^*_e(\mathcal{S})$ such that the following commutative diagram
 holds
 $$\xymatrix{ C^*(\mathcal{S}) \ar@{->>}[dr]^\pi \\
 \mathcal{S} \ar@{^{(}->}[u] \ar@{^{(}->}[r] &
 C^*_e(\mathcal{S})}$$
We let $$p=\pi(\left( \begin{array}{cc}
 1_A & 0 \\
 0 & 0
 \end{array} \right)) \quad \mbox{and} \quad q=\pi(\left( \begin{array}{cc}
 0 & 0 \\
 0 & I_n \otimes 1_A
 \end{array} \right)).$$ Then $p$ and $q$ are projections of
 $C^*_e(\mathcal{S})$ satisfying $p+q=1$ and $pq=0$. Thus we can decompose
  $C^*_e(\mathcal{S})$ in `2 $\times$ 2' matrix corners. Hence
 $\pi$ is corner preserving and there exist $\pi_1,\pi_2,\pi_3,\pi_4$
 such that for any $a \in A$, $b \in M_n(A)$, $x,y \in R_n(A)$,
  $$\pi(\left( \begin{array}{cc}
 a & x \\
 y* & b
 \end{array} \right))=\left( \begin{array}{cc}
 \pi_1(a) & \pi_2(x) \\
 \pi_3(y)* & \pi_4(b)
 \end{array} \right)  .$$
 The (1,2) corners of $\mathcal{S}$ and of $C^*(\mathcal{S})$
 coincide so $\pi_2$ is injective (because $\pi$ extends to $C^*(\mathcal{S})$ the inclusion $\mathcal{S} \subset
 C^*_e(\mathcal{S})$). Similarly $\pi_3$ is injective. On the other hand,
 for any $a \in A$, $x \in R_n(A)$,
 $$\pi_2(ax)=\pi_1(a)\pi_2(x).$$ Thus choosing `good $x$', it shows that $\pi_1$ is injective too.
 Analogously, using $$\pi_2(xb)=\pi_2(x)\pi_4(b), \quad \mbox{for~any}~ b \in M_n(A),~
  x \in R_n(A),$$ the previous argument works to prove the injectivity of $\pi_4$.\\
 Finally, $\pi$ is injective and so $C^*_e(\mathcal{S})=M_{n+1}(A)$. By assumption on $A$, $M_{n+1}(A)$ is an injective
 $C^*$-algebra. Therefore
 $$I(\mathcal{S})=M_{n+1}(A) \quad \mbox{ $*$-isomorphically}$$ and $$I_{11}(R_n(A))=\left( \begin{array}{cc}
 1_A & 0 \\
 0 & 0
 \end{array} \right)I(\mathcal{S})\left( \begin{array}{cc}
 1_A & 0 \\
 0 & 0
 \end{array} \right)=A.$$
This proves $(1)$.
\end{proof}

\begin{remark} We acknowledge that after the paper was submitted,
D. Blecher pointed out to the author a more general result : let
$X$ be an operator space, then for any $p,q \in \mathbb{N}^*$,
$$\mathcal{M}_l(M_{p,q}(X))=M_p(\mathcal{M}_l(X)).$$ We outline the proof. As in
\cite[Paragraph 4.4.11]{BLM1}, we can define the $C^*$-algebra
$\mathcal{C}(X)=I(X)I(X)^*$. Using \cite[Corollary 4.6.12]{BLM1},
we note that
$$\mathcal{C}(M_{p,q}(X))=M_p(\mathcal{C}(X)).$$ Moreover, from
\cite{BP}, the multiplier algebra of $\mathcal{C}(X)$ coincides
with $I_{11}(X)$ i.e. $$\mathcal{M}(\mathcal{C}(X))=I_{11}(X).$$
Hence, using the two previous facts, we can compute
$$\begin{array}{ccll}
               \mathcal{M}_l(M_{p,q}(X))&=&\{ a \in I_{11}(M_{p,q}(X)),~aM_{p,q}(X) \subset M_{p,q}(X) \}\\
               &= & \{ a \in \mathcal{M}(\mathcal{C}(M_{p,q}(X))),~aM_{p,q}(X) \subset M_{p,q}(X) \}\\
               &= & \{ a \in \mathcal{M}(M_p(\mathcal{C}(X))),~aM_{p,q}(X) \subset M_{p,q}(X) \} \\
               &= & \{ a \in M_p(\mathcal{M}(\mathcal{C}(X))),~a_{ij}X \subset X,~ \forall ~i,j \}\\
               &= & \{ a \in M_p(I_{11}(X)),~a_{ij}X \subset X,~ \forall ~i,j \}\\
               &= & M_p(\mathcal{M}_l(X)).
               \end{array}$$
              \end{remark}

 The next theorem enables to represent completely contractively
 a module action on an $n$-minimal operator space into a $C^*$-algebra of the form $C(\Omega,M_n)$.
 It constitutes the main result of this section and generalizes $(i) \Leftrightarrow (iii)$ of
 \cite[Theorem 2.2]{BLM2}.

\begin{theorem}\label{5.3} Let $A$ be a Banach algebra endowed with an operator space structure (resp. a $C^*$-algebra).
Let $X$ be an $n$-minimal operator space which is also a left
Banach $A$-module. Assume that there is a net $(e_t)_t \subset
Ball(A)$ satisfying $e_t \cdot x \to x$, for any $x \in X$. The
following are equivalent :
\begin{enumerate}[(i)]
\item $X$ is a left h-module over $A$. \item There exists a
compact Hausdorf space $\Omega$, a complete isometry $i:X \to
C(\Omega,M_n)$ and a completely contractive homomorphism (resp.
$*$-homomorphism) $\pi:A \to C(\Omega,M_n)$ such that $$i(a \cdot
x)=\pi(a)i(x), \quad \mbox{for~any}~ a \in A,~ x \in X$$
\end{enumerate}
\end{theorem}
\begin{proof} Suppose $(i)$. We first treat the Banach algebra case. By Blecher's oplication Theorem (see
\cite[Theorem 4.6.2]{BLM1}), we know that there is a completely
contractive homomorphism $\eta: A \to \mathcal{M}_l(X)$ such that
$\eta(a)(x)=a \cdot x$, for any $a \in A$, $x \in X$. Using
$\theta$ obtained in Lemma \ref{5.1}, we have a completely
contractive homomorphism $\sigma =\theta \circ \eta: A \to
\mathcal{M}_l(I(X))$ satisfying $$\sigma(a)(x)=a \cdot x,\quad
\mbox{for~any}~ a \in A,~ x \in X.$$ Moreover, $I(X)$ is an
injective $n$-minimal operator space, so
$$I(X)=\oplus^{\infty}_{i \in I} C(\Omega_i,M_{r_i,k_i}) \quad \mbox{completely
isometrically}$$ where the $\Omega_i$'s are Stonean and $r_i,k_i
\leq n$, for any $i \in I$. We have the completely isometric
unital isomorphisms
$$\begin{array}{ccll}
               \mathcal{M}_l(I(X))&=&\oplus^{\infty}_i ~\mathcal{M}_l(C(\Omega_i,M_{r_i,k_i}))\\
               &= & \oplus^{\infty}_i
               ~\mathcal{M}_l(C_{r_i} \otimes_{min} R_{k_i} \otimes_{min}
               C(\Omega_i))\\
               &= & \oplus^{\infty}_i
               ~M_{r_i}(\mathcal{M}_l(R_{k_i} \otimes_{min}
               C(\Omega_i)))& \quad \\
               &= & \oplus^{\infty}_i
               ~M_{r_i}(C(\Omega_i))& \quad \mbox{(by ~Lemma~ \ref{5.2})}\\
               \end{array}$$
               and via these last identifications, the action of $\mathcal{M}_l(I(X))$ on $I(X)$
               is the one inherited from the obvious left
               action of $M_{r_i}$ on $M_{r_i,k_i}$. More precisely for any $u=(f_i \otimes y_i )_i
                \in \mathcal{M}_l(I(X))$ and $x=(g_i \otimes x_i )_i
                \in I(X)$, $$u (x)=(f_ig_i \otimes y_ix_i )_i.$$ For each $i$, let
               $\varphi_i:M_{r_i} \to M_n$ (resp. $\phi_i:M_{r_i,k_i} \to
               M_n$) be the embedding of $M_{r_i}$ (resp.
               $M_{r_i,k_i}$) in the `up-left corner' of $M_n$.
               Hence, as in the end of
the proof of Corollary \ref{3.3},  we have now a $*$-isomorphism
$$\begin{array}{ccll}
               \psi ~:& \mathcal{M}_l(I(X))&\to &C(\Omega,M_n)\\
               &(f_i \otimes y_i )_i & \mapsto & \sum_i \tilde{f}_i \otimes
               \varphi_i(y_i)
               \end{array}$$ and a complete isometry $$\begin{array}{ccll}
               j ~:& I(X)&\to &C(\Omega,M_n)\\
               &(g_i \otimes x_i )_i & \mapsto & \sum_i \tilde{g}_i \otimes
               \phi_i(x_i)
               \end{array}$$ which verify $$j(u(x))=\psi(u)j(x) \quad \mbox{for~any}~ u \in \mathcal{M}_l(I(X))
               , ~ x \in I(X)$$ Finally $\Omega$, $i=j_{\vert X}$ and $\pi=\psi \circ \sigma$ satisfy
               the desired relations. If $A$ is a $C^*$-algebra, we conclude using the fact that a contractive homomorphism
               between $C^*$-algebras is necessarily a $*$-homomorphism.
\end{proof}

\begin{remark}\label{5.4} \begin{enumerate}[(1)]
\item From the previous result, a $C^*$-algebra which acts
`suitably' on an $n$-minimal operator space is necessarily an
extension of a subhomogeneous $C^*$-algebra of degree $\leq n$.
\item Suppose that $A$ is unital and its action too (i.e. $1 \cdot
x=x$ for any $x$ in $X$). In the previous result, we cannot expect
to obtain a unital completely contractive homomorphism $\pi$.
Because when $A$ is an operator algebra and $A=X$, the assumption
$(i)$ is verified (see the BRS theorem \cite[Theorem
2.3.2]{BLM1}). Hence this particular case leads back to the Remark
\ref{3.4}.
\end{enumerate}
\end{remark}

The theorem below could be considered as an `$n$-minimal version'
of the CES-theorem (see \cite[Theorem 3.3.1]{BLM1}). It is the
bimodule version of Theorem \ref{5.3} and its proof is
`symmetrically' the same using the two lemmas above.

\begin{theorem}\label{5.5} Let $A$ and $B$ be two Banach algebras endowed with an operator space structure (resp. two $C^*$-algebras).
Let $X$ be an $n$-minimal operator space which is also a Banach
$A$-$B$-bimodule. Assume that there is a net $(e_t)_t \subset
Ball(A)$ (resp. $(f_s)_s \subset Ball(B)$) satisfying $e_t \cdot x
\to x$ (resp. $x \cdot f_s \to x$), for any $ x \in X$. The
following are equivalent :
\begin{enumerate}[(i)]
\item $X$ is an h-bimodule over $A$ and $B$. \item There exists a
compact Hausdorf space $\Omega$, a complete isometry $i:X \to
C(\Omega,M_n)$ and two completely contractive homomorphisms (resp.
$*$-homomorphisms) $\pi:A \to C(\Omega,M_n)$ and $\theta:B \to
C(\Omega,M_n)$ such that
$$i(a \cdot x \cdot b)=\pi(a)i(x)\theta(b), \quad \mbox{for~any}~a \in
A,~ b \in B,~ x \in X.$$
\end{enumerate}
\end{theorem}

The next result states that if $A$ and $B$ are originally
$n$-minimal operator algebras, then $\pi$ and $\theta$ can be
chosen completely isometric. This corollary generalizes
\cite[Corollary 2.10]{BLM2}.

\begin{corollary}\label{5.6} Let $A$, $B$ and $X$ be three $n$-minimal
operator spaces such that $A$ and $B$ are approximately unital
operator algebras and $X$ is a Banach $A$-$B$-bimodule. Assume
that there is a net $(e_t)_t \subset Ball(A)$ (resp. $(f_s)_s
\subset Ball(B)$) satisfying $e_t \cdot x \to x$ (resp. $x \cdot
f_s \to x$), for any $ x \in X$. The following are equivalent :
\begin{enumerate}[(i)]
\item $X$ is a left h-module over $A$. \item There exists a
compact Hausdorf space $\Omega$, a complete isometry $i:X \to
C(\Omega,M_n)$ and completely isometric homomorphisms $\pi:A \to
C(\Omega,M_n)$ and $\theta:B \to C(\Omega,M_n)$ such that $$i(a
\cdot x \cdot b)=\pi(a)i(x)\theta(b), \quad \mbox{for~any}~a \in
A,~ b \in B,~ x \in X.$$
\end{enumerate}
\end{corollary}
\begin{proof} From Theorem \ref{5.5}, there exists a compact
Hausdorf space $K_0$, a complete isometry $j:X \to C(K_0,M_n)$ and
completely contractive homomorphisms $\pi_0:A \to C(K_0,M_n)$ and
$\theta_0:B \to C(K_0,M_n)$ satisfying $$j(a \cdot x \cdot
b)=\pi_0(a)i(x)\theta_0(b),$$
for any $a \in A$, $b \in B$, $x \in
X$. Moreover by Corollary \ref{2.4}, there exists a compact
Hausdorf space $K_A$ (resp. $K_B$) and a completely isometric
homomorphism $\pi_A:A \to C(K_A,M_n)$ (resp. $\theta_B:B \to
C(K_B,M_n)$). Let $$C=C(K_A,M_n) \oplus^{\infty} C(K_0,M_n)
\oplus^{\infty} C(K_B,M_n) = C(\Omega,M_n)$$ where $\Omega$ is the
disjoint union of $K_A,K_B$ and $K_0$. Let $i:X \to C(\Omega,M_n)$
defined by $i(x)=0 \oplus j(x) \oplus 0$, for any $x \in X$ so $i$
is a complete isometry. Let $\pi:A \to C(\Omega,M_n)$ (resp.
$\theta:B \to C(\Omega,M_n)$) defined by $\pi(a)= \pi_A(a) \oplus
\pi_0(a) \oplus 0$, for any $a \in A$ (resp. $\theta(b)= 0 \oplus
\theta_0(b) \oplus \theta_B(b)$, for any $b \in B$ ). Hence, $\pi$
and $\theta$ are completely isometric homomorphisms. Finally,
$\Omega$, $\pi$, $\theta$ and $i$ satisfy the desired relation.
\end{proof}

\end{document}